\documentclass[10pt]{article}

\usepackage{amsmath,amssymb,color,amsthm}
\usepackage{enumerate,enumitem}
\usepackage{booktabs,tabularx,multirow}
\usepackage{graphicx,subfigure}
\usepackage{verbatim,mathrsfs}
\usepackage[margin=3 cm]{geometry}

\DeclareGraphicsExtensions{.pdf}
\graphicspath{{figs/}}

\theoremstyle{plain}
\newtheorem{theorem}{Theorem}
\newtheorem{lemma}[theorem]{Lemma}

\newtheorem{conj*}[theorem]{Conjecture}

\newtheorem*{claim*}{Claim}

\theoremstyle{definition}

\theoremstyle{remark}

\newtheorem*{remark*}{Remark}

\usepackage[hidelinks]{hyperref}

\usepackage[utf8]{inputenc} 
\usepackage[maxnames=4,firstinits=true,bibencoding=utf8,backend=bibtex,isbn=false]{biblatex}
\newcommand{\cS}{\mathcal{S}}

\newcommand{\cR}{\mathcal{R}}
\newcommand{\cC}{\mathcal{C}}
\newcommand{\cF}{\mathcal{F}}
\newcommand{\cB}{\mathcal{B}}
\newcommand{\cA}{\mathcal{A}}
\newcommand{\A}{\mathcal{A}}

\newcommand{\cX}{\mathcal{X}}

\begin{document}

\title{Boolean lattices: Ramsey properties and embeddings.}
 \author{ Maria Axenovich and   Stefan  Walzer \thanks{Karlsruhe Institute of Technology}}

 \maketitle

 \begin{abstract}
 A subposet   $Q'$ of a poset $Q$ is  a  {\it copy of a poset} $P$   if there is a bijection  $f$ between elements of $P$ and $Q'$ such that 
 $x\leq y$ in $P$  iff $f(x)\leq f(y)$ in $Q'$. 
 For  posets $P, P'$, let the {\it poset Ramsey number} $R(P,P')$ be the smallest $N$ such that no matter how the elements of the Boolean lattice $Q_N$ are colored red and blue, 
 there is a copy of $P$ with all red elements or a copy of $P'$ with all blue elements.  
We provide some general bounds on $R(P,P')$ and focus on the situation when $P$ and $P'$ are both Boolean lattices.
In addition, we give asymptotically tight bounds for the number of copies of $Q_n$ in $Q_N$ and for 
a multicolor version of a poset Ramsey number. 
 \end{abstract}

 \section{Introduction}

 A classical hypergraph Ramsey number $R(G,H)$ for $k$-uniform hypergraphs $G$ and $H$ is the smallest $n$ such that any red/blue edge-coloring 
  of a complete $k$-uniform hypergraph on $n$ vertices contains a red copy of $G$ or a blue copy of $H$. The existence of this number was proved by \textcite{Ramsey} in 1930, but  the problem of determining these and “multicolor”  numbers remains open and generates a lot of research activity, see for example  \cite{Conlon09, Conlon10, Conlon13, Cooley, AGLM, KMV}. 
 
 While classical Ramsey theory provides results about unavoidable uniform set systems such as $k$-uniform hypergraphs,  here we investigate Ramsey numbers for set systems with various set sizes treated as partially ordered sets.

  A {\it partially ordered set}, a {\it poset}, is a set with an accompanying relation “$\leq$”  that is transitive, reflexive, and antisymmetric. A {\it Boolean lattice} of dimension $n$, denoted  $Q_n$,  is the power set of an $n$-element ground  set $X$ equipped with inclusion relation, we also write $2^X$ for such a lattice. The $k$-th layer of $Q_n$ is the set of all $k$-element subsets of the ground set, $0 \leq k \leq n$.
 An injective mapping $f: P  \rightarrow P'$ is an {\it embedding} of a poset $P$ into another poset $P'$
 if for any $A, B\in  P$ we have  $A\leq  B$ in $P$ iff $f(A)\leq f(B)$ in $P'$.  In this case, we say that $f(P)$ is a  {\it 
copy}  of $P$ in $P'$.  The { \it  $2$-dimension} of  a poset $P$, defined by \textcite{Trotter75} and denoted  by $\dim_2(P)$,  is the smallest $n$ such that $Q_n$ contains a copy of $P$. An $n$-element antichain $A_n$  is a poset with any two elements not comparable,  an $n$-element chain $C_n$ is a poset with any two elements comparable.

 A very general problem of determining what posets have \emph{Ramsey property} was solved by   \textcite{NR}.
 Here, a poset $X$ has Ramsey property if for any poset $P$ there is a poset $Z$ such that when one colors the copies of $X$ in $Z$ red or blue, there 
 is a copy of $P$ in $Z$  such that all copies of $X$ in this copy of $P$ are red or all of them are blue. In the current  paper, $X$ is always the single-element poset, i.e., the elements of posets are colored, instead of more complicated substructures. In that case, the existence of $Z$ is guaranteed, so we direct our attention to quantitative aspects and the asymmetric case where the two colors are associated with different posets.

 Let $P$ and  $P'$ be two posets and $g$ be a poset parameter, 
 then $R_g(P,P')$ is the smallest $g'$ such that there is a poset $Z$ with $g(Z)=g'$ and such that in any red/blue coloring of $Z$ there is a copy of $P$ with all red elements or a copy of $P'$ with all blue elements. 
 \textcite{Trotter86}   considered $g$ to be width, height,  or size,  see  also \cite{TrotterSurvey, trotter-book, Paoli85}.  In this paper we initiate the study of Ramsey numbers for posets that extends the  previously defined notions and 
fits the concept of extremal functions of posets \cite{Patkos15, Methuku14}. 
 
For  posets  $P$ and $P'$, let the \textbf{\textit{poset  Ramsey number}}  $R_{\dim_2}(P, P')$  be the smallest $N$ such that any red/blue-coloring of $Q_N$ contains either a red copy of $P$ or a blue copy of $P'$.  We use simply $R(P, Q)$ instead of $R_{\dim_2}(P, Q)$ when it is clear from context.

For example,  $R(C_n, C_n)= 2n -2$ for a  chain $C_n$. Indeed, in a $2$-colored $Q_{2n-2}$ there is a chain with $2n-1$ elements, 
$n$ of those have the same color and form a copy of $C_n$;  a coloring of $Q_{2n-3}$ with $n-1$ layers in one color and $n-1$ layers in another color has no monochromatic copy of $C_n$. Another example  is  $R(A_n,A_n)= \min\{N: ~ 2n -1\leq \binom{N}{\lfloor N/2 \rfloor}\}$, for an  antichain $A_n$. To see that, observe that  at least half of  the elements in the middle layer in a $2$-colored $Q_N$ have the same color and there are at least $n$ of them. 
On the other hand, partition the elements of $Q_{N-1}$ into at most $2n-2$ chains. Such a partition exists, see for example \cite{vanLint}. 
Color $n-1$ chains red and the  remaining (at most $n-1$) chains blue. Since each antichain has at most one element from each of these chains, there is no monochromatic antichain with $n$ elements.
The focus of this paper is the case when $P$ and $P'$ are Boolean lattices themselves, i.e., $R(Q_n, Q_m)$.

 \begin{theorem} \label{thm:cube}
 For any integers $n, m \geq 1$, 
 \begin{enumerate}
 \renewcommand{\labelenumi}{(\roman{enumi})}  
 \item  $2n  \leq R(Q_n, Q_n) \leq n^2 +2 n$,
 \item  $R(Q_1, Q_n) = n+1$, 
 \item  $R(Q_2, Q_n) \leq  2n+2$,
 \item  $R(Q_n, Q_m) \leq mn+n+m, $
 \item  $R(Q_2, Q_2)= 4$,  ~$R(Q_3, Q_3) \in \{7,8\}$,
 \item  A Boolean lattice $Q_{3n \log(n)}$ whose elements  are colored red or blue randomly and independently  with equal probability  contains  a monochromatic  copy of $Q_n$ asymptotically  almost surely.
 \end{enumerate}
  \end{theorem}

We prove a more general result in terms of the $2$-dimension and  height $h$ of a poset $P$, where
$h(P)$ is the number of elements in a largest chain of $P$.
The {\it  lexicographic  product}  $P\times Q$ of two posets is defined by taking  a disjoint union of $|P|$  copies of $Q$ and 
 making all elements in $i$th copy less than each element in $j$th copy iff  the $i$th element of $P$ is less than $j$th  element of $P$.
Formally,  the set of elements of   $P\times Q$  is a cartesian product of elements of $P$ and $Q$ and 
$(p_1, q_1)\leq (p_2, q_2)$ in    $P\times Q$     iff $p_1\leq p_2$ or  ($p_1=p_2$ and $q_1\leq q_2$). Note that the product operation is not symmetric, i.e., in general $P \times Q$ and $Q \times P$ are not isomorphic.
We say that a coloring of $2^S$ is {\it layered} for a set $S$ if for each $i$, $i=0, \ldots, |S|$ all subsets of $S$ of size~$i$ have the same color.

\begin{lemma} [Layered Lemma] \label{layer}
For any $n$ there is $N$ such that in any red/blue coloring $c$ of subsets of~$[N]$ there is a set $S\subseteq [N]$ such that $|S|=n$ and $c$ is layered on
$2^S$.  
\end{lemma}

 \begin{lemma}[Blob Lemma]\label{blob}
 Let $P$ be a poset and $m\geq 1$ be an integer. If $N= dim_2(P) + h(P)m $, then $Q_N$ contains a copy of  $P\times Q_m$.
 In particular, $R(P, Q_m) \leq dim_2(P) + h(P)m$. 
 \end{lemma}

 \begin{lemma}[Antichain Lemma]\label{antichains}
 If in a red/blue coloring of $2^{[N]}$ the red elements form  antichains $\cA_1, \ldots, \cA_{\ell}$, $\ell<N$,  such that 
 $A_i$ is a set of minimal elements in $\cA_i\cup \cdots \cup \cA_{\ell}$, $i=1, \ldots, \ell$, then  
 there is a blue copy of $2^{[N-\ell]}$.
 \end{lemma}
 
\noindent We give some more definitions and prove Lemmas \ref{layer},  \ref{blob}, \ref{antichains}   in Section~\ref{definitions}.
We prove Theorem~\ref{thm:cube} in Section~\ref{cube}.

In studying Ramsey properties of posets, specifically Boolean lattices, we prove a result of  independent interest concerning  embeddings of $Q_n$ into $Q_N$.  We describe a bijective mapping between the set of such embeddings and  a set of some special sequences in Section \ref{enumerate}. 
As a corollary, we determine  the number of such embeddings.

  \begin{theorem}\label{enum}
 Let $n \leq N$. The number $e(n,N)$ of embeddings of $Q_n$ into $Q_N$ satisfies the inequality
 $$ \frac{N!}{(N-n)!} (a(n)-n)^{N-n} \leq e(n, N) \leq   \frac{N!}{(N-n)!} a(n)^{N-n}, $$ where $a(n)$ is the number of distinct antichains in $2^{[n]}$.
 In particular 
  $$  e(n, N) =   2^{\binom{n}{\lfloor \frac{n}{2} \rfloor }  (N-n)(1+o(1))}.$$
  \end{theorem}

The {\it multicolor Ramsey number}  for a poset $P$ is defined, for a given $k$,  to be the smallest integer $n$ such that any coloring of $Q_n$ in $k$ colors contains a copy of $P$ in one of the colors. This number is denoted by  $R_k(P)$. The following theorem is proved in   Section \ref{Multicolor}.

\begin{theorem}\label{multicolor}
For any poset $P$ that is not an antichain,   $R_k(P)= \Theta(k)$. 
\end{theorem}

Finally, we make a quick comparison between Ramsey numbers for Boolean lattices and Ramsey numbers for Boolean algebras.
A Boolean algebra  $B_n$  of dimension $n$  is 
 a set system $\{ X_0 \cup  \bigcup_{i\in I} X_i:    I\subseteq [n]\},$
 where $X_0, X_1, \ldots, X_n$ are pairwise disjoint sets, $X_i \neq \emptyset$ for $i=1, \ldots, n$.
 To see the difference between Boolean algebras  and copies of  $Q_n$ consider three families of sets in $2^{[6]}$:  
$ \cF_1 = \{  \{2\}, $   $\{2,3\},$   $  \{2, 4, 5\}, $   $ \{2,3,4, 5, 6\}\}$, 
$ \cF_2 = \{ \{2\}, \{2,3,4\}, \{2, 5\}, $ $ \{2,3, 4, 5\}\}$,
 $\cF_3 = \{ \{2\}, \{2,3\},$$ \{2, 3, 5\}, \{2,3, 4, 5\}\}.$
 Here $\cF_2$ is a Boolean algebra of dimension $2$ with $X_0= \{2\}, X_1= \{3,4\}$, $X_2=\{5\}$,  but $\cF_1$ and $\cF_3$ are not.
The families $\cF_1$ and $\cF_2$ are copies of $Q_2$, however $\cF_3$ is not a copy of $Q_2$ because in $Q_2$ there are two incomparable elements and in $\cF_3$ any two elements are comparable.
 Boolean algebras have  much more restrictive  structure than Boolean lattices. If a subset of $Q_N$ contains a Boolean algebra of dimension $n$, then
 it clearly has a copy of $Q_n$, but not necessarily the other way around. 
 \textcite{Gunderson99} considered the number $R_{\mathsf{Alg}}(n)$, defined to be the smallest $N$ such that any red/blue coloring of subsets of an $N$ element set contains a red or a blue Boolean algebra of dimension $n$.    Note that here “contains” means not a  subposet containment  but simply a subset containment in  $2^{[N]}$. The existence of $R_{\mathsf{Alg}}(n)$ and thus of  $R_{\dim_2}(Q_n, Q_n)$  easily follows from Lemma \ref{layer}. However, the bounds   given by this approach are very large.
We state the bounds on $R_{\mathsf{Alg}}(n)$,  here  the lower bound is given without a proof by  \textcite{Brown-Erdos} and the upper bound is a recapturing of the arguments given by  \textcite{Gunderson99}. Here,  $K^n( s, \ldots, s)$ is  a complete $n$-uniform $n$-partite hypergraph with parts of size $s$ each and 
 $R_h(K^n(2, \ldots, 2))$ is  be the smallest $N'$ such that any $2$-coloring of  $K^n(N', \ldots, N')$ 
contains a monochromatic  $K^n(2, \ldots, 2)$. Here $h$ stands for  hypergraph.

\begin{theorem} \label{Algebra}
There is a positive constant $c$ such that
$$2^{cn} \leq   R_{\mathsf{Alg}}(n)\leq \min\{ 2^{2^{n+1} n \log n},  n R_h(K^n(2, \ldots, 2))\}.$$ 
\end{theorem}

We put the details of the  upper bound   in Section \ref{boolean-algebra}.  The last section we devote to conclusions.


\section{ Definitions and proofs of lemmas}\label{definitions}

For a subset  $U$  of elements in  a poset $P$, the {\it upper set} of $U$ or simply the {\it upset}~of $U$,  is $U^+ = \{ x \in P\mid ~  \exists u\in U\colon   u\leq x\}$. If $U = U^+$, then $U$ is called upper closed.
 We say that a sequence of sets is {\it  inclusion respecting} if any higher indexed set is not   a subset of  a lower indexed set.
 Note that one can always create such an  ordering for any family of sets by putting the sets of smaller cardinality ahead of 
 sets of larger cardinality.

An {\it affine cube}, or a {\it Hilbert cube} of dimension $n$  is a set of integers $\{ x_0 +  \sum_{i\in I} x_i:    I\subseteq [n]\},$
 where $x_0 $ is a  non-negative integer and $x_1, \ldots, x_n$ are  positive integers.
 Note that if  $\mathcal A = \{ X_0 \cup \bigcup_{i\in I} X_i:    I\subseteq [n]\}$ is a Boolean algebra and $x_i=|X_i|$, $i=0, \ldots, n$, 
 then the set of all sizes of sets occurring in $\A$ is equal to  $\{ |X| : X \in \mathcal A\} = \{ x_0 +  \sum_{i\in I} x_i:    I\subseteq [n]\}$ and therefore is a Hilbert cube.  Hilbert \cite{Hilbert} showed that for  any positive integers $c$ and $n$  there is $N=h(n,c)$ such that any $c$-coloring of $[N]$ contains a monochromatic Hilbert cube of dimension $n$. 
Let $b(N,n)$ be  the largest size of a family of subsets of $[N]$ that does not contain 
  a Boolean algebra of dimension $n$. 
  
The {\it Lubell mass}  $\ell$ of a set family $\cF\subseteq 2^{[N]}$  is 
$\ell(\cF) = \sum_{S\in \cF}   \binom{N}{|S|} ^{-1}$. Note that $\ell(2^{[N]}) = N+1$.
For a poset $P$, let $\lambda^*(P) = \lim \sup _{N\rightarrow \infty} \{\ell(\cF): ~ \cF\subseteq 2^{[N]} \mbox{  contains no copy of } P\}$. 
M\'eroueh  \cite{Meroueh15} proved that $\lambda^*(P)$ exists for any $P$.

For other definitions on graphs, posets, and set systems, we refer the reader to West and Trotter  \cite{west_introduction_2000,Trotter86}.

\begin{proof}[Proof of Layered Lemma]
 Let $R^{k}(n)$ be the hypergraph Ramsey number that is equal to the smallest integer $N$ such that 
 any red/blue coloring of $\binom{[N]}{k}$  contains a monochromatic complete hypergraph $\binom{S}{k}$ for $|S|=n$.
Let $c$ be a red/blue coloring of $2^{[N]}$ for $N= R^1(R^2(R^3(\cdots   (R^{n-1}(n)) \cdots )))$.
We shall prove by induction on $i$, $i=1, \ldots, n-1$ that there is a set $S_i\subseteq [N]$ such that $\binom{S_i}{j}$ is monochromatic for each $j = 0, \ldots, i$, 
$|S_i|= R^{i+1}(R^{i+2}(\cdots   (R^{n-1}(n)) \cdots )))$, and $S_i\supseteq  S_{i+1}$.
Since $N=  R^1(R^2(R^3(\cdots   (R^{n-1}(n)) \cdots )))$,    a coloring  $c$ restricted to  $1$-element sets  contains a monochromatic family
of size  $R^2(R^3(\cdots   (R^{n-1}(n)) \cdots ))$. Let this set be $S_1$.
Assume that $S_1, \ldots, S_{i-1}$ satisfying the assumption exist. Since $|S_{i-1}| = R^{i}(R^{i+1}(\cdots   (R^{n-1}(n)) \cdots ))$,
  a coloring $c$ restricted to subsets of size $i$ in $S_{i-1}$ contains  a set $S_i\subseteq S_{i-1}$ all of whose $i$-element subsets are of the same color and such that 
$|S_i| = R^{i+1}(R^{i+2}(\cdots   (R^{n-1}(n)) \cdots ))$.  This completes the induction. We have that $|S_{n-1}|=  R^n(n) = n$, so  
 $2^{S_{n-1}}$ is a Boolean lattice of dimension $n$ with  a layered coloring.
\end{proof}

 \begin{proof}[Proof of Blob Lemma]
 Let $n= dim_2(P)$ and let $f$ be an embedding of $P$ into $2^{[n]}$.
 Let  $h(P)=h$ and $N = n+ hm$. Consider a partition of $[N]  = X_0 \cup X_1 \cup \ldots \cup X_h$, where 
 $X_0 = [n]$ and $|X_i| = m$, $i=1, \ldots, h$.  Let $f_i$ be an isomorphism from  $2^{[m]}$ to $2^{X_i}$, $i=1, \ldots, h$.
 We construct an embedding $g$ of $P\times 2^{[m]}$ into $2^{[N]}$. For  $p \in P$ and $S\subseteq [m]$,
 let $h(p)$ be the number of elements in a  longest chain in $P$ with maximum element $p$ and let 
$g((p, S)) = f(p) \cup X_1\cup \ldots \cup X_{h(p)-1}  \cup f_{h(p)}(S) $, where   $h(p)\geq 1$. 
 
 By definition of the product, we have that $(p,S)$ is less than $(p', S')$ in $P\times 2^{[m]}$ iff $p$ is less than $p'$ in $P$  or ( $p=p'$  and $S\subseteq S'$) and it is straightforward to check that $g$ respects this relation and is therefore an embedding.
 
 For the second claim, consider a coloring of $2^{[N]}$ in red and blue.  Consider further a copy of $P\times 2^{[m]}$ in it. 
 If some copy of $2^{[m]}$ is blue, we are done, otherwise there is a red element $(p,S_p)$ for each $p\in P$ and some $S_p\subseteq [m]$.
 These red elements form a red copy of $P$.
 \end{proof}

 \begin{proof}[Proof of Antichain Lemma]
 
 Consider  $\cX_0 =  2^{\{\ell+1, \ldots, N\}}$. Let $n= N-\ell$.
 We shall show a stronger statement by induction on $i$, $i=0, \ldots, \ell$ :   there is a copy  $\cX_i$  of $2^{[n]}$ with sets from $\{1, 2, \ldots, i,  \ell+1, \ldots, N\}$   such that 
 this copy does not contain any set from $\cA_1, \ldots, \cA_i$.
 
 Basis: The claim is trivial for $i = 0$ since $\cX_0$ fulfills the requirement as no sets are forbidden yet.

 Step:  Assume that $\cX_{i-1}$ is a copy of $2^{[n]}$ consisting of subsets of $\{1, 2, \ldots, i-1,  \ell+1, \ldots, N\}$  and 
 avoiding all of $\cA_1, \cA_2, \ldots, \cA_{i-1}$, $i\geq 1$.    Let $\cX_i$ be the image of an embedding $f$  of $\cX_{i-1}$, where 
 $f(Y)= Y$ if $Y\not \in \cA_i^+$, $f(Y)= Y\cup \{i\}$ if $Y\in \cA_i^+$.  As before it is an  embedding as it preserves inclusions.
  We only need to check that $f(Y) \not \in \cA_1, \cA_2, \ldots, \cA_i$ for any $Y\in \cX_{i-1}$.
 Assume that $f(Y) \in \cA_j$ for some $j\in \{1, \ldots, i\}$.  
 Since any $Y$ from $\cX_{i-1}$ is not in $\cA_1, \ldots, \cA_{i-1}$ and the elements of $\cA_i$ in $\cX_{i-1}$ were shifted up by $\{i\}$, we see that 
 if $f(Y)\in \cA_j$, then  $Y\in \cA_i^+$, and  so $f(Y) = Y\cup \{i\}$. Thus, there is $B\in \cA_i$, such that $B\subseteq Y$. So, $B\subseteq Y\subseteq Y\cup\{i\} = C$ for some 
 $C\in \cA_j$. If $i=j$ then we have two comparable sets $B$ and $C$ in an antichain $\cA_i$, a contradiction. 
So,  $j< i$. But $\cA_j$ consists of minimal elements of $\cA_j\cup \cdots \cup \cA_\ell$ and  $C\in \cA_j$ is not a minimal element, a contradiction.
\end{proof}
 
 \section{Ramsey numbers for Boolean lattices} \label{cube}

    \begin{proof}  [Proof of Theorem \ref{thm:cube} (i)]
   For the lower bound on $R(Q_n,Q_n)$, consider  $Q_{2n-1}$.  Color the sets of sizes   $0, \ldots, n-1$ red and all other 
   sets  blue. Then there is no monochromatic chain with $n+1$ elements and thus there is no monochromatic copy of $Q_n$.
    
   For the upper bound on $R(Q_n, Q_n)$, consider  a red/blue coloring of $Q_{n^2+2n}$.
   Let  the ground set be  $ X_0 \cup X_1  \cup \ldots \cup X_{n+1}$, where $X_i$'s are pairwise disjoint and of size $n$ each.
   Consider families of sets $\cB_Y$ for each  $Y\subseteq X_0$ with $|Y|\geq 1$  to be 
  $\cB_Y = \{ Y \cup X_1 \cup \ldots \cup X_{|Y|} \cup X:~  X\subseteq X_{|Y|+1}\}$, 
  let $\cB_\emptyset  =2^ {X_1}$. 
    We see that each $\cB_Y$ is  a copy of $Q_n$.  If   this copy is blue, then $\cB_Y$ gives a monochromatic 
    copy of $Q_n$. Otherwise,  there is a red element in each $\cB_Y$. This element is $Z_Y=  Y \cup X_1 \cup \ldots \cup X_{|Y|} \cup S_Y$,  where $S_Y\subseteq X_{|Y|+1}$.
    We claim that these elements form a red copy of $Q_n$. Indeed, we see for $Y, Y' \subseteq [n]$ that 
    $Y\subseteq Y'$ iff $Z_Y\subseteq Z_{Y'}$.   
    \end{proof}

     \begin{proof}  [Proof of Theorem \ref{thm:cube} (ii)]
    To show that $R(Q_1, Q_n) = n+1$,  consider a red/blue coloring of $Q_{n+1}$. 
    If there are two red sets $A\subseteq B$, then they form a red copy of $Q_1$ and we are done. 
    So, we can assume that the red sets form an antichain, $\cA$.
    From Lemma \ref{antichains} we see that there is a copy of $Q_n$ avoiding $\cA$.  
    To see that $R(Q_1, Q_n)>n$, color one element of $Q_n$ red and other blue.
    \end{proof}

     \begin{proof}  [Proof of Theorem \ref{thm:cube} (iii)]
     To show that $R(Q_2, Q_n) \leq 2n+2$,  consider a red/blue coloring of $Q_{2n+2}$. 
    Consider the poset formed by  red elements. Assume that it has height at least  $n+3$. Then there are two sets $A\subseteq B$, $|B\setminus A|\geq n+2$.
    Let $\cX=\{ Y:    A\subseteq Y\subseteq B\}$. Then $\cX$ is a  copy of $Q_m$, $m\geq n+2$,  with maximal and minimal elements red.
    This implies that all red elements in this poset  form a chain, otherwise there is a red  copy of $Q_2$. 
   Let  $a$ be contained in all sets of this chain but not in $A$.  Let $b\in B\setminus (A\cup \{a\})$.
  Consider $\cX' = \{ Y\in \cX:  a\not \in Y,  b\in Y\}$.  
    Then $\cX'$ is a blue copy of $Q_{m-2}$ containing a blue copy of $Q_n$. 
    Thus, we can assume that the height $h$ of the red poset $R$  is at most $n+2$. Build an antichain $\cA_1$ to be the set of minimal elements of $R$ and 
    an antichain $\cA_i$, $i>1$ to be the set of minimal elements of $R\setminus (\cA_1\cup \cdots \cup \cA_{i-1})$,  $i=1, \ldots, h$.
    Applying Lemma \ref{antichains}, we see that  there is a blue copy of $Q_{2n+2-(n+2)}= Q_n$.
    \end{proof}

       \begin{proof}  [Proof of Theorem \ref{thm:cube} (iv)]
    The bound on $R(Q_n,Q_m)$ is a corollary of the Blob Lemma.
    Note  that our upper bound on  $R(Q_n, Q_n)$  also follows from the Blob Lemma, but we wrote an explicit direct proof.
    \end{proof}
    
     \begin{proof}[Proof of Theorem \ref{thm:cube} (v)]
     To prove that $R(Q_2, Q_2) \geq 4$, consider a layered coloring of $Q_3$ with two red layers and two blue layers. This coloring has no monochromatic copy of $Q_2$. 
     For the upper bound, consider a red/blue coloring of $2^{[4]}$ and treat two cases: when $\emptyset$ and $[4]$ have the same or different colors. 
     Relatively easy case analysis shows that in both cases there is a monochromatic copy of $Q_2$.
     
    To prove that $R(Q_3, Q_3)\in \{7,8\}$ we first give an explicit coloring of $2^{[6]}$ containing no monochromatic copy of $Q_3$. 
    We list the family $\cR$ of  red sets in Figure \ref{fig:q6withoutq3}. All other sets are blue, denote their family by $\cB$.
    Note that a set $S$ is  in $\cR$ iff  $[6]\setminus S$ is  in $\cB$. So, it is sufficient to verify that $\cR$ contains no copy of $Q_3$.

    \begin{figure}[htb]
        \centering
        \includegraphics{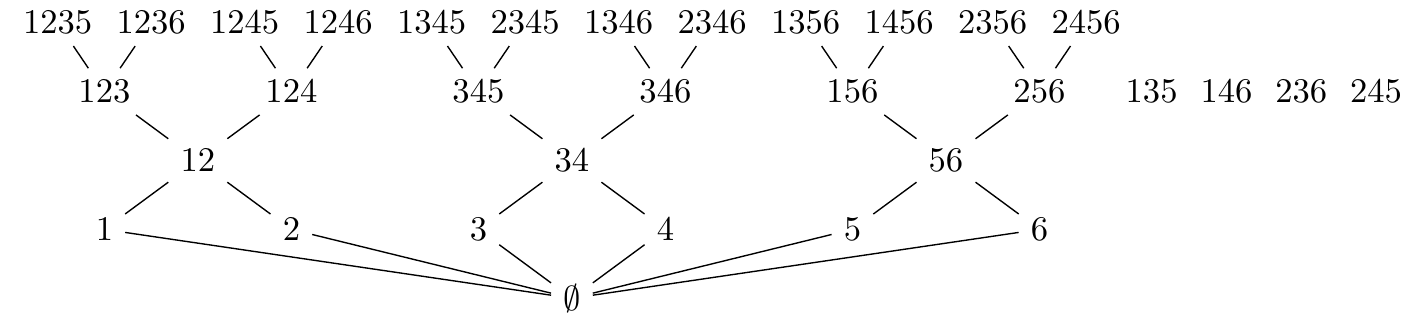}
     \caption{The family $\cR$ of red subsets of $[6]$ that  avoid a  copy of~$Q_3$. Braces and commas are omitted. The lines illustrate inclusion, but inclusions involving the four sets on the right are ommitted to keep the picture simple.}    \label{fig:q6withoutq3}
    \end{figure}

   Assume for the contrary that $f: 2^{[3]}\rightarrow \cR$ is an embedding. Let $Y=f([3])$.    Clearly, $Y$ should have size $3$ or $4$.  
We assume that $|Y|=4$, since replacing  $Y$  with one of its superset in $\cR$ still 
 gives a copy of $Q_3$.  Without loss of generality $Y=\{1,2,3,5\}$. 
Then $2^Y\cap \cR = \{\{1,2, 3, 5\}, \{1,2, 3\}, $ $ \{1,3, 5\},$ $  \{1,2\},  \{1\},$ $ \{2\}, \{3\}, \{5\}, \emptyset \}$, that clearly does not contain a copy of $Q_3$.
 
    
    To show that $R(Q_3, Q_3)\leq 8$, consider a red/blue coloring of $2^{[8]}$. 
    We say that a family of five distinct sets  $A_1, \ldots, A_5$   forms a configuration $\cA$ if either $A_1, \ldots, A_5$ or  their complements   satisfy the following conditions:
     $A_1\subset  A_2, A_3, A_4\subset A_5$,  $|A_5\setminus A_i| \geq 5$  for $i= 2, 3, 4$, $|A_2|=|A_3|=|A_4|$,  and $|A_i\cup A_j| = |A_i|+1$, $2\leq i< j\leq 4$. 
    We treat four  cases and  in each of them we find either a monochromatic copy of $Q_3$ or a monochromatic family $\cA$.   
    
     Case 1. The elements  of some four  layers have the same color.  Then there is a monochromatic copy of $Q_3$. 
    
    Case 2.  The elements of  some three consecutive  layers are of the same color, say red, and Case 1 does not hold.  
    If  there is  a red element above or below the three  red layers,  then this element and some $7$ members of the three red layers form a red copy of $Q_3$.
    If there is no such element, then there are $6$ blue layers, that brings us to Case 1.

    Case 3.  The elements of the first two or the last two layers have the same color, say  the first two layers are red and Cases 1 and 2 do not hold. 
 Since last three layers contain both red and blue elements, let $M$ be a set of size at least $6$ that is red. If $a, b, c \in M$, then $\emptyset, \{a\}, \{b\}, \{c\}, M$ form a monochromatic configuration $\cA$.
 
 Case 4.  None of the Cases 1, 2, 3 hold.   Let $\emptyset$   be red. Then, w.l.o.g. $\{1\}$ is blue and by pigeonhole principle $\{1,2\}, \{1,3\}, \{1,4\}, \{1,5\}$ are all of the same color  $t\in \{$red, blue$\}$.
 If either $[8]$ or $[8]\setminus \{i\}$ is of color $t$, $i=2, 3, 4, 5$, then we have a monochromatic configuration $\cA$. Otherwise $[8],  [8]\setminus \{2\}, [8]\setminus \{3\}, [8]\setminus \{4\}, B$, where $B\in \{\emptyset, \{1\}\}$ form a 
 monochromatic configuration $\cA$. 
 
 It remains to show that if there is a monochromatic configuration $\cA =\{A_1, A_2, A_3, A_4, A_5\}$ then there is a monochromatic copy of $Q_3$. 
 Assume w.l.o.g. that $|A_5\setminus A_i| \geq 5$ for $i=2, 3, 4$.
 Further assume that $A_i = A\cup \{i\}$, for $i=2, 3, 4$ and some $A$.   Consider $B_k =\{ S:    A\cup \{i, j\}  \subseteq S\subseteq A_5 \setminus \{k\} \}$ for $\{i,j,k\} = \{2,3,4\}$.
 Then each $B_k$, $k=2, 3, 4$ is a copy of $Q_3$. If at least one of this copies is monochromatic, we are done. Otherwise each $B_k$ has a red element, $C_k$. 
 Then $A_1, A_2, A_3, A_4, C_1, C_2, C_3, A_5$ form a monochromatic copy of $Q_3$.
 \end{proof}

 \begin{proof}[Proof of Theorem \ref{thm:cube} (vi)]
 Here, we do not try to optimize the constant $3$. In fact, it could be replaced with $1+ \epsilon$ for any positive $\epsilon$.
   Let $N= n+ (n+1)2 \log n $. Let $[N] = X_0\cup X_1\cup \cdots \cup X_{n+1}$, where $X_i$'s are pairwise disjoint, $|X_0|=n$ and 
 $|X_i| = m$, $i=1, \ldots n+1$, where $m= 2 \log n$.  We drop floors and ceilings  here and assume that $m$ is even.
 Color subsets of $[N]$ with red and blue such that a set gets color red with probability $1/2$ and the sets are colored independently.
 For a set $S\subseteq X_0$, let $\cF(S) = \{ S\cup X_1 \cup \cdots \cup X_{|S|} \cup X:   X\subset X_{|S|+1}, |X| = m/2\}$.
 The probability that $\cF(S)$ has only blue sets is $2^{- \binom{m}{m/2}}$. 
 The probability that each $\cF(S)$ has a red set for each $S$ is thus
 at least $p(n) = 1 -  2^n 2^{- \binom{m}{m/2}}$. Here, we use a simple union bound and the fact that there are $2^n$ choices for $S$.
 Since $2^n2^{-\binom{m}{m/2}} = 2 ^{n - \binom{m}{m/2}} \leq  2^{n - n - \log n } \xrightarrow[n\to\infty]{}  0$, we have that 
   $p(n) \xrightarrow[n\to  \infty] {} 1$. Therefore, asymptotically almost surely each $\cF(S)$ has a red set $F_S$. 
 In that case $\{F_S:  ~S\subseteq X_0\}$ is a copy of $Q_n$.
 \end{proof}


 \section{\texorpdfstring{Embeddings of $Q_n$ into $Q_N$.}{Embeddings of Qn into QN.}} \label{enumerate}

In this section we define two bijections between the set of copies of $Q_n$ in $Q_N$ and sets of other combinatorial objects. We use these bijections to prove Theorem \ref{enumerate} and to make 
some general observations about Ramsey numbers for Boolean lattices.

\subsection{Bijection from the set of good sequences}

 Let $f: 2^{[n]} \rightarrow 2^{[N]}$ be an embedding.
 Consider the sequence of sets $(U_1(f), \ldots, U_N(f))$ with $$U_j (f)= \{S\in 2^{[n]}:  j\in f(S)\}, \quad j \in [N].$$
 We call this the \emph{characteristic vector} of the embedding $f$. An example is shown in Figure \ref{fig:fibres}. 
 We call a sequence of upper-closed sets from $2^{[n]}$ {\it good}  if  each of $(\{i\})^+$, $i=1, \ldots, n$ appear in that sequence.
 Note that some $U_j$s could be empty.
\begin{figure}[htb]
    \centering
    \includegraphics{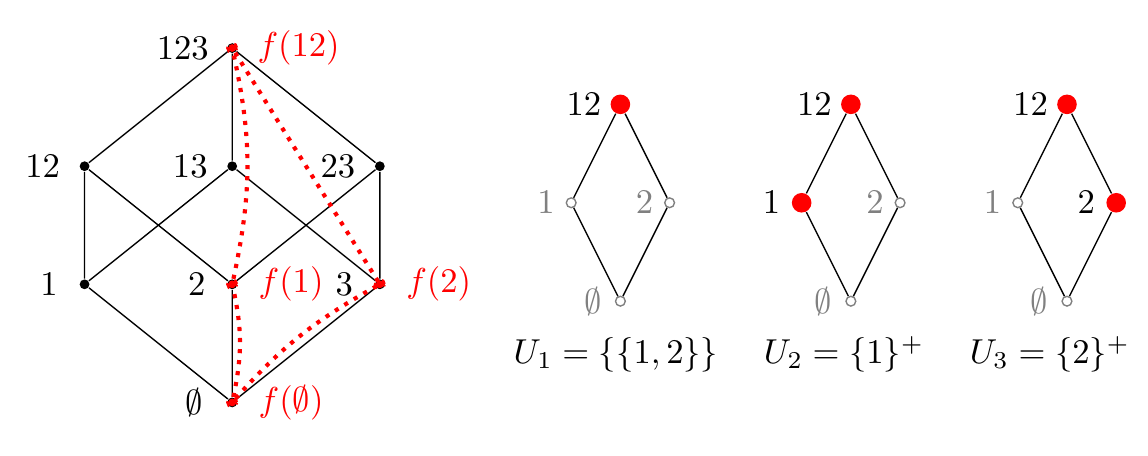}
    \caption{Example for an embedding $f$ of $Q_2$ into $Q_3$ with corresponding characteristic vector $(U_1,U_2,U_3)$.}
    \label{fig:fibres}
\end{figure}
If, for example $U_3 = \{1\}^+$,  $U_7= \{3\}^+$, and $U_8= \{2\}^+$ in an embedding $f$ from $2^{[3]} $ to $2^{[9]}$, 
then any subset of $\{3, 7, 8\}$ is equal to    $f(S) \cap \{3, 7, 8\}$, for some $S$.   For example $\{7, 8\}  \subseteq f(\{3, 2\})$ since  $3 \not\in f(\{3, 2\})$,  it follows that 
$f(\{3, 2\} )\cap \{3, 7, 8\} = \{7, 8\}$.

  \begin{theorem}\label{fibers}
  There is a bijection between the set of embeddings of $2^{[n]}$ into  $2^{[N]}$ and the set of good sequences of  $N$ upper-closed  sets from $2^{[n]}$.
  The bijective mapping assigns each such embedding its characteristic vector. 
   \end{theorem}
  
  \begin{proof}
 
{\it  Claim 1~~}  Let $(U_1, \ldots, U_N)$ be the characteristic vector of an embedding  $f: 2^{[n]} \rightarrow 2^{[N]}$.  Then for each $j=1, \ldots,N$, the set $U_j$ is upper closed and for each $i=1,\ldots,n$ there is a $U_j$ such that $U_j= \{i\}^+.$\\

 To see that each $U_j$ is upper closed, consider $S\in U_j$ and $S'\supseteq S$. 
 Then $j\in f(S)$, and $f(S) \subseteq f(S')$. Thus $j\in f(S')$, and so $S'\in U_j$. To prove the second statement, note that 
for any $ i\in [n]$ we have $\{i\} \not\subseteq [n]\setminus \{i\}$. Thus  $f(\{i\})\not\subseteq f([n]\setminus \{i\})$, implying that there is  $ j $ such that  $ j\in f(\{i\} )\setminus f([n]\setminus \{i\}).$
Now consider any $S\subseteq [n]$.  If $i\in S$, then $\{i\} \subseteq S$, so $j\in f(\{i\})\subseteq f(S)$.
 If $i\not \in S$, then $S \subseteq [n] \setminus \{i\}$, and because of 
 $f(S) \subseteq f([n] \setminus \{i\})$ and $j \notin f([n] \setminus \{i\})$ we conclude $j \notin f(S)$.
 Thus $ i\in S \Leftrightarrow  j\in f(S) $. By definition we also have $j\in f(S) \Leftrightarrow S\in U_j$ and $i\in S \Leftrightarrow S\in \{i\}^+$.
Thus $S\in  \{i\}^+ \Leftrightarrow S\in U_j$. This implies that   $U_j=\{i\}^+$, proving Claim 1.\\

{\it Claim 2 ~~}
If  $(U_1, \ldots, U_N)$ is a sequence of  upper closed sets   and 
 all $\{i\}^+$ occur among the $U_j$s, for $i=1, \ldots, n$, then the function $f$ with $f(S) =  \{j:  S\in U_j\}$ is an embedding of $Q_n$ into $Q_N$.
 Moreover, $(U_1, \ldots, U_N)$ is the characteristic vector of $f$.\\

First we verify that $f$ is indeed an embedding. 
 If $S\subseteq T$ then since $U_j$'s are upper closed sets,  if $S\in U_j$ then $T\in U_j$, so
 $f(S)\subseteq f(T)$. On the other hand, if $S\not \subseteq T$ then there is $i\in S\setminus T$.
 Let $j$ be an index such that $\{i\}^+ = U_j$. Then  $j \in f(S) \setminus  f(T)$ and thus $f(S) \not\subseteq f(T)$. 
 Now, to check that $(U_1, \ldots, U_N)$ is the characteristic vector of $f$, recall 
 that $U_j(f) =  \{S\subseteq [n] :  j\in f(S)\}$. On the other hand $f(S) = \{j:  S\in U_j\}$, so $U_j=\{ S\subseteq [n]: j \in f(S)\}$.
 Thus $U_j=U_j(f)$, for all $j \in [N]$.  This proves Claim 2.\\

{\it  Claim 3~~}
 If $f$ and $g$ are distinct embedding of $2^{[n]}$ into $2^{[N]}$ then they have distinct characteristic vectors, i.e., 
 $(U_1(f),  \ldots, U_N(f)) \neq (U_1(g), \ldots, U_N(g))$. \\

 Since $f\neq g$ we have $S\subseteq [n]$ with $f(S)\neq g(S)$, i.e., there is, without loss of generality $i\in f(S)\setminus g(S)$.
 Thus $S\in U_i(f)$, but $S\not\in U_i(g)$, which proves Claim 3.\\
 
 Claims  1, 2, and 3 show that the characteristic vector provides the desired bijection.
 \end{proof}

   \subsection{ Bijection from  the set of unions using inclusion preserving map}

   \begin{theorem}
   A set $\cS\subseteq 2^{[N]} $ forms  a copy of $2^{[n]}$ if and only if 
   there is $I \subseteq   [N]$ of size $n$ and an  inclusion preserving map $\phi:  2^{I} \rightarrow 2^{[N]\setminus I}$
   such that $\cS = \{Y\cup \phi(Y):   Y\subseteq I\}$. 
     \end{theorem}

   \begin{proof}

   Assume that there are such $I$ and $\phi$, so  that $\cS =   \{  Y\cup \phi(Y):   Y\subseteq I \}$. We need to show that $\cS$ induces a copy of $2^{[n]}$ in $2^{[N]}$. 
   We shall show that $f: 2^I\rightarrow  2^{[N]}$, where $f(Y) = Y\cup \phi(Y)$ is an embedding.  For this we need to observe  that for $Y, Y'\subseteq I$ we have $Y\subseteq Y'$ 
   iff $Y\cup \phi(Y)\subseteq Y'\cup \phi(Y')$. 
   Indeed, if $Y \subseteq Y'$,  then   since $\phi$ is inclusion preserving, $\phi(Y)\subseteq \phi(Y')$, so $Y\cup \phi(Y)\subseteq Y'\cup \phi(Y')$. 
   If $Y\cup \phi(Y)\subseteq Y'\cup \phi(Y')$, and $Y, Y'\subseteq I$, $\phi(Y), \phi(Y')\subseteq [N]\setminus I$,   we have that $Y\subseteq Y'$.

The other way around, assume that $\cS$ is a copy of $2^{[n]}$ in $2^{[N]}$ obtained as a result of an embedding~$f$. 
From Theorem \ref{fibers}, we know that the characteristic vector $(U_1, \ldots, U_N)$ of $f$ is good, i.e., there is an index set $I$ of size $n$ such that 
$U_{g(i)} = \{i\}^+$ for $i\in [n]$ and $g$ is a bijection from $[n]$ to $I$. 
This means that $\{f(S) \cap I : S \subseteq [n]\} = 2^I$. Since $|\cS| = 2^{|I|}$, there is for each $Y\subseteq I$ a unique $\phi(Y)\subseteq [N] \setminus I$ such that 
some set $S\in \cS$ is equal to $Y\cup \phi(Y)$. 
It remains to show that $\phi$ is inclusion preserving. 
Assume that for some $Y, Y'\subseteq I$, $Y\subseteq Y'$ and $\phi(Y)\not\subseteq \phi(Y')$. 
Then the corresponding sets $Y\cup \phi(Y)$ and $Y'\cup \phi(Y')$ are not comparable.
Thus the number of comparable pairs  is $\cS$ is strictly less that the number of comparable pairs in $2^I$.
This is a contradiction since $\cS$ forms an copy of $2^I$.
\end{proof}

   \subsection{Counting copies of $Q_n$}
   \begin{proof}[Proof of Theorem \ref{enum}]
   It follows from Theorem \ref{fibers}, that the number $e(n,N)$ of copies of $Q_n$ in $Q_N$ is equal to the number of good sequences. 
 Each upper closed set is uniquely determined by an antichain of its minimal elements. For the upper bound, 
 observe that there are $\binom{N}{n}$ ways to choose the positions of $\{1\}^+, \ldots, \{n\}^+$ in the sequence.
 There are $n!$ ways to place those in these positions. The remaining $N-n$ positions could be occupied by any upper closed set, so 
 there are $a(n)^{N-n}$ ways to choose them. All together there are at most  $\binom{N}{n}n! a(n) ^{N-n} = \frac{N!}{(N-n)!} a(n)^{N-n}$ ways to form such a good sequence.
 For the lower bound, we count only the good sequences with each $\{i\}^+$ appearing exactly once. 
 There are again $\frac{N!}{(N-n)!} $ ways to place $\{1\}^+, \ldots, \{n\}^+$ in $n$ positions in the sequence and there 
 are $(a(n)-n)^{N-n}$ ways to place upper sets different from these in the remaining $N-n$ positions.
 
 Note that one could provide an exact formula for $e(N, n)$ in terms of $a(n)$   by counting the words of length $N$  over the alphabet $[a(n)]$, 
 containing each of the  letters $1, 2, \ldots, n$ at least once.

 Finally, to prove the last statement of the Theorem, observe that
 $a(n) = 2^{\binom{n}{\lfloor \frac{n}{2} \rfloor }(1+ O(\log n/n))}$, see~\cite{Kleitman75}, 
 and $\frac{N!}{(N-n)!} \leq  N^n  = 2^{n\log N} $.
 \end{proof}

 \section{Multicolor Ramsey numbers}\label{Multicolor}
 
 \begin{proof}[Proof of Theorem \ref{multicolor}]
 Consider a coloring $c$ of $Q_N$ in $k$ colors and having no monochromatic copy of $P$. Then each color class contains no copy of $P$ and thus 
 Lubell mass of each color class $C_i$ is $\ell(C_i) \leq \lambda^*(P)$. 
 Now, from the definition $\ell(2^{[N]}) = \sum_{i=1}^{k} \ell(C_i) \leq k \lambda^*(P)$. 
 On the other hand, $\ell(2^{[N]}) = \sum_{i=0}^{N}  \binom{N}{i} \binom{N}{i}^{-1} = N+1$. 
 Thus $N+1 \leq k \lambda^*(P) = \Theta(k)$. 
 
 To show that $R_k(P) \geq \Omega(k)$, consider a layered copy of $Q_{k-1}$ with each  layer of own color. Then each color class is an antichain, thus does not contain a copy of  $P$. 
 So, $R_k(P) \geq k$.
 \end{proof}

 \section{Boolean algebras} \label{boolean-algebra}

 While     \textcite{Gunderson99}  consider  a multicolor Ramsey problem for Boolean algebras when the dimension of the 
 desired monochromatic Boolean lattice is fixed and the number of colors grows,  we consider  a $2$-colored case. 
 Note that Layered Lemma \ref{layer}  immediately gives a very large upper bound on $R_{\mathsf{Alg}}(n)$  in terms of repeated application of hypergraph Ramsey's Theorem and Hilbert's Theorem.   Indeed, we know from Lemma \ref{layer} that sufficiently large $2$-colored Boolean lattice   contains  a  layered Boolean algebra $B$ of dimension $N'$ with sets in layer $i$ of size $i$. 
Consider a coloring $c'$ of $\{0, 1, \ldots, N'\}$, where $c'(i)$ is the color of elements in layer $i$, $i= 0, \ldots, N'$.
If $N'> h(n,2)$ then there is a monochromatic, say red,  Hilbert cube of dimension $n$  in $c'$, say  $\{ x_0 +  \sum_{i\in I} x_i:    I\subseteq [n]\}$.
This means that all subsets of $S_n$ of sizes $x_0,   x_0+x_1, x_0+x_2, \ldots, x_0+x_n,     x_0+x_1+x_2, \ldots,  x_0+x_1 + \cdots + x_n$  are red. 
Pick disjoint subsets  $X_0, X_1, \ldots, X_n$ of  $B$  of sizes $x_0, \ldots, x_n$ respectively. We can do this since $N'\geq x_0+x_1 + \cdots + x_n$.
Then these sets generate a monochromatic Boolean algebra of dimension $n$. \\

 \begin{proof}[Proof of Theorem \ref{Algebra}]
 
 To get the  first expression in the upper bound,
consider a bound on $b(N,n)$, the largest  number of elements in $Q_N$ not containing  a Boolean algebra of dimension $n$,   given in \cite{Gunderson99}:
 $$b(N, n) \leq 10^{n} 2^{-\frac{1}{2^{n-1}}} n^{n-\frac{1}{2^n}} \cdot N^{-\frac{1}{2^n}} 2^N.$$  
 Thus if $N= 2^{2^{n+1} n \log n}$ then there is a color class of size at least $2^N/2 > b(N,n)$. Therefore there is a monochromatic Boolean algebra of dimension $n$. 
 
 To get the second  expression in the upper bound, we use the ideas from  \cite{Gunderson99} again. 
 Consider a coloring $c$ of $Q_N$. Split $[N]$ into $n$  pairwise disjoint sets $N_1, \ldots, N_n$ of almost equal sizes, 
 further consider $\cC_i$ to be a longest chain in a family of subsets of $N_i$,  $i=1, \ldots, n$.
 Consider a complete  $n$-uniform $n$-partite hypergraph with parts $\cC_i$, $i=1,\ldots, n$.
 Color a hyperedge $\{X_1, X_2, \ldots, X_n\}$ with color $c(X_1\cup \cdots \cup X_n)$. 
 If $N/n \geq R_h(K^n(2, \ldots, 2))$ then we have a monochromatic $K^n(2, \ldots, 2)$ with 
 partite sets $\{X_1,  X_1'\}, \ldots, \{X_n, X_n'\}$, where $X_i\subseteq X_i'$ and $X_i' \cap X_j' = \emptyset$ for $i\neq j$.  This corresponds to a monochromatic Boolean algebra with sets  $Y_0\cup \{ \cup_{i\in I} Y_i: I\subseteq [n]\}$, 
 where $Y_0= X_1 \cup \cdots \cup X_n$ and $Y_i = X_i'\setminus X_i$, $i=1, \ldots, n$.

 For the lower bound, we use the layered coloring of $Q_N$ with $N= h(n,2)-1$ in red and blue, so that 
 the indices of blue layers form a set without affine cube of dimension $n$ and so do the indices of the red layers.
 It was noticed in \cite{Brown-Erdos}, that it is easy to see that $h(n,2)\geq 2^{cn}$. 
   \end{proof}

 \section{Conclusions}\label{Conclusions}
 In this paper, we initiated the study of Ramsey theory for copies of posets in Boolean lattices. 
 Compared with at least  exponential behavior of a similar Ramsey number for Boolean algebras, 
 we obtain the bounds $ 2n-1 \leq R(Q_n, Q_n)\leq n^2 +2n$. 
 We give a tight asymptotical expression of the number of embeddings of $Q_n$ into $Q_N$ and
show that the multicolor Ramsey number for posets is linear in the number of colors.

 \noindent
 The question of determining the correct value for $R(Q_n, Q_n)$ remains one of the interesting here.
 The fact that $R(Q_3, Q_3)\in \{7,8\}$ suggests that neither of our bounds for $R(Q_n, Q_n)$ is tight.
 Another interesting question is how the colorings of $Q_N$, for $N=R(Q_n, Q_n)-1$ and no monochromatic
 copies of $Q_n$ look like. So far, for all lower bounds on Boolean lattices and Boolean algebras, only layered colorings were 
 considered. Here, we presented a non-layered coloring of $Q_6$ with no monochromatic copy of $Q_3$. Note that any 
 layered coloring of $Q_6$ has a monochromatic $Q_3$.   For some additional Ramsey-type results in the Boolean lattice, see 
 \textcite{walzer}.
 
%
%
%
%
%

 \section{Acknowledgements}  The authors thank Tom Trotter, Dwight Duffus, Kevin Milans,   and  David Gunderson for inspiring discussion. Particular thanks go to 
 Dwight Duffus for bringing  \cite{Gunderson99} to author's attention and mentioning different sorts of embeddings. The authors thank also Torsten Ueckerdt and Jonathan Rollin 
 for discussions.

  \printbibliography

\end{document}